\documentclass[12pt]{amsart}

\newtheorem{theorem}{Theorem}[section]
\newtheorem{lemma}{Lemma}[section]

\theoremstyle{definition}
\newtheorem{definition}{Definition}[section]
\newtheorem{corollary}{Corollary}[section]
\newtheorem{example}[theorem]{Example}

\theoremstyle{remark}

\numberwithin{equation}{section}

\begin{document}
\title[Pointwise bi-slant submanifolds of Kenmotsu manifolds]{Warped product pointwise bi-slant submanifolds of Kenmotsu manifolds}
\author[S. K. Hui, J. Roy and T. Pal]{Shyamal Kumar Hui$^*$, Joydeb Roy and Tanumoy Pal}
\keywords{Kenmotsu manifold, warped product, pointwise bi-slant submanifold.\\$*$ Corresponding author.}

\begin{abstract}
The present paper deals with the study of warped product pointwise bi-slant submanifolds of Kenmotsu manifolds with an example. The characterization for such submanifold is also discussed. An inequality of such submanifold is obtained and its equality case is also considered.
\end{abstract}
\maketitle
\section{Introduction}
In \cite{TANNO} Tanno classified connected almost contact metric manifolds whose automorphism groups possess the maximum dimension.
For such a manifold, the sectional curvature of plane sections containing $\xi$ is a constant, say $c$, which that they could
be divided into three classes:
(i) homogeneous normal contact Riemannian manifolds with $c > 0$,
(ii) global Riemannian products of a line or a circle with a K\"{a}hler manifold of constant holomorphic sectional curvature
          if $c=0$ and
(iii) a warped product space $\mathbb{R} \times _f \mathbb{C}^n$ if $c< 0$.
It is known that the manifolds of class (i) are characterized by admitting a Sasakian structure.
The manifolds of class (ii) are characterized by a tensorial relation admitting a cosymplectic structure. Kenmotsu \cite{KEN}
characterized the differential geometric properties of the manifolds of class (iii) which are nowadays called Kenmotsu
manifolds and later studied by several authors (see \cite{HUI1}, \cite{HUI2}, \cite{HUI3} and references therein).\\
\indent In 1990, Chen \cite{CHENS} introduce the notion of slant submanifolds in a Hermitian manifold. Then Lotta \cite{LOTTA} has defined and studied
slant immersions of a Riemannian manifold  into an almost contact metric manifold. As a natural generalization of slant
submanifolds Etayo \cite{ETAYO} defined the notion of pointwise slant subamnifolds under the name of quasi-slant submanifolds.
Pointwise slant submanifolds in almost contact metric manifold is studied in (\cite{MIUD},\cite{PARK}). On the other hand, Carriazo \cite{CAR1}
defined and studied bi-slant submanifolds in almost Hermitian manifolds and simultaneously gave the notion of pseudo-slant submanifolds. Then
Khan and Khan \cite{6} studied contact version of pseudo-slant submanifolds. Bi-slant submanifolds of Kenmotsu space form is studied in \cite{PGS}.\\
\indent As a generalization of Riemannian product manifold Bishop and O'Neill \cite{BISHOP} defined and studied warped product manifolds. The study of warped product submanifolds was initiated by Chen (\cite{CHENCR1}, \cite{CHENCR2}). Then many authors studied warped product submanifolds
of different ambient manifold (\cite{CHENBOOK}, \cite{HAN}-\cite{HUOM}, \cite{UD1}). In this connection it may be mention that warped product submanifolds of Kenmotsu manifold are studied in (\cite{ATCE1}, \cite{SOLAMY}, \cite{OTHMAN}, \cite{KHAN}, \cite{KHANS}, \cite{UDD2}, \cite{UDD1}, \cite{UKK}, \cite{UOO}).\\
\indent Recently, Chen and Uddin \cite{CHENUD} studied pointwise bi-slant submanifolds of Kaehler manifold. Also, Khan and Shuaib \cite{KHANS}
studied pointwise pseudo-slant submanifolds in Kenmotsu manifolds. Motivated by the above studies the present paper deals with the study of
warped product pointwise bi-slant submanifolds of Kenmotsu manifolds. The paper is organized as follows. Section 2 is concerned with some preliminaries.
Section 3 deals with the study of pointwise bi-slant submanifolds of Kenmotsu manifold. In section 4, we have studied warped product pointwise bi-slant submanifolds of Kenmotsu manifold
with an interesting example. In section 5, we have found the characterization for warped product pointwise bi-slant submanifolds of Kenmotsu manifolds. Section 6 is concerned with an inequality on warped product pointwise bi-slant submanifold of Kenmotsu manifold whose equality case is also considered.
\section{Preliminaries}
An odd dimensional smooth manifold $\bar{M}^{2m+1}$ is said to be an almost contact metric manifold \cite{BLAIR} if it admits a $(1,1)$
tensor field $\phi$, a vector field $\xi$, an $1$-form $\eta$ and a Riemannian metric $g$ which satisfy
\begin{equation}\label{2.1}
  \phi \xi=0,\ \ \ \eta(\phi X)=0, \ \ \ \phi^2 X=-X+\eta(X)\xi,
\end{equation}
\begin{equation}\label{2.2}
  g(\phi X,Y)=-g(X,\phi Y), \ \ \ \eta(X)=g(X,\xi), \ \ \ \eta(\xi)=1,
\end{equation}
\begin{equation}\label{2.3}
  g(\phi X,\phi Y)=g(X,Y)-\eta(X)\eta(Y)
\end{equation}
for all vector fields $X,Y$ on $\bar{M}$.\\
\indent An almost contact metric manifold $\bar{M}^{2m+1}(\phi, \xi, \eta, g)$ is said to be Kenmotsu manifold if the following conditions hold \cite{KEN}:
\begin{equation}\label{2.4}
  \bar{\nabla}_X \xi=X-\eta(X)\xi,
\end{equation}
\begin{equation}\label{2.5}
 (\bar{\nabla}_X \phi)(Y)=g(\phi X ,Y)\xi-\eta(Y)\phi X,
\end{equation}
where $\bar{\nabla}$ denotes the Riemannian connection of $g$.\\
\indent Let $M$ be an $n$-dimensional submanifold of a Kenmotsu manifold $\bar{M}$. Throughout the paper we assume that
the submanifold $M$ of $\bar{M}$ is tangent to the structure vector field $\xi$.\\
\indent Let $\nabla$ and $\nabla ^\bot$ be the induced connections on the tangent bundle $TM$ and the normal bundle $T^\bot M$ of $M$ respectively.
Then the Gauss and Weingarten formulae are given by
\begin{equation}\label{2.6}
  \bar{\nabla}_X Y=\nabla_X Y+h(X,Y)
\end{equation}
and
\begin{equation}\label{2.7}
  \bar{\nabla}_XV=-A_V X+\nabla _X^ {\bot}V
\end{equation}
for all $X,Y\in \Gamma(TM)$ and $V\in \Gamma(T^\bot M)$, where $h$ and $A_V$ are second fundamental form and the shape operator
(corresponding to the normal vector field $V$) respectively for the immersion of $M$ into $\bar{M}$. The second fundamental form
$h$ and the shape operator $A_V$ are related by $g(h(X,Y),V)=g(A_V X,Y)$
for any $X,Y\in \Gamma(TM)$ and $V\in \Gamma(T^\bot M)$, where g is the Riemannian metric on $\bar{M}$ as well as on $M$.

The mean curvature $H$ of $M$ is given by $H=\frac{1}{n}\text{trace}\ h$. A submanifold of a Kenmotsu manifold $\bar{M}$ is said to be totally umbilical if $h(X,Y)=g(X,Y)H$ for any $X,Y\in \Gamma(TM)$. If $h(X,Y)=0$ for all $X,Y\in \Gamma(TM)$, then $M$ is totally geodesic and if $H=0$ then $M$ is minimal in $\bar{M}$.\\

Let $\{e_1,\cdots,e_n\}$ be an orthonormal basis of the tangent bundle $TM$ and $\{e_{n+1},\cdots,e_{2m+1}\}$ an orthonormal basis of the normal bundle
 $T^\bot M$. We put
 \begin{equation}\label{2.8a}
 h_{ij}^r=g(h(e_i,e_j),e_r)\ \text{and} \ \ \|h\|^2=g(h(e_i,e_j),h(e_i,e_j)),
 \end{equation}
 for $r\in\{n+1,\cdots,2m+1\}$.\\
 \indent For a differentiable function $f$ on $M$, the gradient $\boldsymbol{\nabla}f$ is defined by
 \begin{equation}\label{2.8b}
 g(\boldsymbol{\nabla}f,X)=Xf,
 \end{equation}
 for any $X\in\Gamma(TM)$. As a consequence, we get
 \begin{equation}\label{2.8c}
 \|\boldsymbol{\nabla} f\|^2=\sum_{i=1}^{n}(e_i(f))^2.
 \end{equation}
\indent For any $X\in \Gamma(TM)$ and $V\in \Gamma(T^\bot M)$, we can write
\begin{eqnarray}
\label{2.8}
\text{(a)} \  \phi X = PX+QX,\ \  \text{(b)}\  \phi V= bV+cV
\end{eqnarray}
where $P X,\ bV$ are the tangential components and $QX,\ cV$ are the normal components.\\
A submanifold $M$ of an almost contact metric manifold $\bar{M}$ is said to be invariant if $\phi(T_pM)\subseteq T_pM$, for every $p\in M$ and
anti-invariant if $\phi T_pM\subseteq T^\bot_pM$, for every $p\in M$ \cite{BEJ}.\\
A submanifold $M$ of an almost contact metric manifold $\bar{M}$ is said to be slant if for each non-zero vector $X\in T_pM$, the angle $\theta$ between $\phi X$
and $T_pM$ is a constant, i.e. it does not depend on the choice of $p\in M$.
\begin{definition}\cite{MIUD}
A submanifold $M$ of an almost contact metric manifold $\bar{M}$ is said to be pointwise slant if for any non-zero vector $X\in T_pM$ at $p\in M$,
such that $X$ is not proportional to $\xi_p$, the angle $\theta(X)$ between $\phi X$ and $T_p^*M=T_pM-\{0\}$ is independent of the choice of non-zero
$X\in T_p^*M$.
\end{definition}
For pointwise slant submanifold $\theta$ is a function on $M$, which is known as slant function of $M$. Invariant and anti-invariant submanifolds are particular cases of pointwise slant submanifolds with slant function $\theta=0$ and $ \frac{\pi}{2}$ respectively. Also a pointwise slant submanifold $M$ will be slant if and only if $\theta$ is constant on $M$, Thus a pointwise slant submanifold is proper if neither $\theta=0,\frac{\pi}{2}$ nor constant.
\begin{theorem}\cite{MIUD}
Let $M$ be a submanifold of an almost contact metric manifold $\bar{M}$ such that $\xi\in \Gamma(TM)$. Then, $M$ is pointwise slant if and
 only if
 \begin{equation}\label{2.9}
 P^2=\cos^2\theta(-I+\eta\otimes\xi),
 \end{equation}
 for some real valued function $\theta$ defined on the tangent bundle $TM$ of $M$.
\end{theorem}
If $M$ be a pointwise slant submanifold of an almost contact metric manifold $\bar{M}$ the following relations holds:
\begin{eqnarray}
\label{2.10}  g(PX,PY) &=& \cos^2\theta\{g(X,Y)-\eta(X)\eta(Y)\} \\
\label{2.11}  g(QX,QY) &=& \sin^2\theta\{g(X,Y)-\eta(X)\eta(Y)\}\\
\label{2.12}  bQX&=&\sin^2\theta\{-X+\eta(X)\xi\}, \ \ \ cQX=-QPX
\end{eqnarray}
for any $X,\ Y\in\Gamma(TM)$.\\
We now recall the following:
\begin{theorem}\cite{HIPKO} [Hiepko's Theorem].
Let $\mathcal{D}_1$ and $\mathcal{D}_2$ be two orthogonal distribution on a Riemannian manifold $M$. Suppose that $\mathcal{D}_1$ and $\mathcal{D}_2$ both are involutive such that $\mathcal{D}_1$ is a totally geodesic foliation and $\mathcal{D}_2$ is a spherical foliation. Then $M$ is locally isometric to a non-trivial warped product $M_1\times_fM_2$, where $M_1$ and $M_2$ are integral manifolds of $\mathcal{D}_1$ and $\mathcal{D}_2$, respectively.
\end{theorem}
\section{Pointwise bi-slant submanifolds}
 In this section, we define and
 study pointwise bi-slant submanifold of Kenmotsu manifold.
 \begin{definition}\cite{CHENUD}
   A submanifold $M$ of an almost contact metric manifold $\bar{M}$ is said to be a pointwise bi-slant submanifold if there exists a pair of orthogonal distribution $\mathcal{D}_1$ and $\mathcal{D}_2$ of $M$ at the point $p\in M$ such that\\
   (a) $TM=\mathcal{D}_1\oplus\mathcal{D}_2\oplus\{\xi\}$,\\
   (b) $\phi(\mathcal{D}_1)\bot\mathcal{D}_2\oplus\{\xi\}$,\\
   (c) the distribution $\mathcal{D}_1$ and $\mathcal{D}_2$ are pointwise slant with slant functions $\theta_1$ and $\theta_2$, respectively.
 \end{definition}
 The pair $\{\theta_1,\theta_2\}$ of slant functions is called the bi-slant function and if $\theta_1,\theta_2\neq 0,\frac{\pi}{2}$ and
 both $\theta_1,\theta_2$ are not constant on $M$, then $M$ is called proper pointwise bi-slant submanifold.\\
 If $M$ is a pointwise bi-slant submanifold of $\bar{M}$ then for any $X\in \Gamma(TM)$ we have
 \begin{equation}\label{3.1}
 X=T_1X+T_2X
 \end{equation}
 where $T_1$ and $T_2$ are the projections from $TM$ onto $\mathcal{D}_1$ and $\mathcal{D}_2$, respectively.\\
 \indent If we put $P_1=T_1\circ P$ and $P_2=T_2\circ P$ then from (\ref{3.1}), we get
 \begin{equation}\label{3.2}
 \phi X= P_1X+P_2X+QX,
 \end{equation}
 for $X\in \Gamma(TM)$. From Theorem 2.1, we get
 \begin{equation}\label{3.3}
 P_i^2X=\cos^2\theta_i\{-X+\eta(X)\xi\},\ \ X\in \Gamma(TM), \ \ i=1,2.
 \end{equation}
 For a proper pointwise bi-slant submanifold $M$ of a Kenmotsu manifold $\bar{M}$, the normal bundle of $M$ is decomposed as
 \begin{equation}\label{3.3a}
 T^\bot M=Q\mathcal{D}_1\oplus Q\mathcal{D}_2\oplus\nu,
 \end{equation}
 where $\nu$ is a $\phi$-invariant normal subbundle of $M$.\\
 Now, we find the following lemma for a pointwise bi-slant submanifold $M$ of Kenmotsu manifold $\bar{M}$.
 \begin{lemma}
 Let $M$ be a pointwise bi-slant submanifold of $\bar{M}$ with pointwise slant distributions $\mathcal{D}_1$ and $\mathcal{D}_2$ with distinct slant function $\theta_1$ and $\theta_2$, respectively. Then
 \begin{eqnarray}
 \label{3.4}  (\sin^2\theta_1-\sin^2\theta_2)g(\nabla_XY,Z) =g(A_{QP_2Z}Y-A_{QZ}P_1Y,X)\\
 \nonumber +g(A_{QP_1Y}Z-A_{QY}P_2Z,X)
\end{eqnarray}
and
\begin{eqnarray}
  \label{3.5} (\sin^2\theta_2-\sin^2\theta_1)g(\nabla_ZW,X) =g(A_{QP_2W}X-A_{QW}P_1X,Z)\\
  \nonumber +g(A_{QP_1X}W-A_{QX}P_2W,Z)-\eta(X)g(Z,W),
 \end{eqnarray}
 for any $X,\ Y\in \Gamma(\mathcal{D}_1\oplus\{\xi\})$ and $Z,\ W\in\Gamma(\mathcal{D}_2)$.
 \end{lemma}
 \begin{proof}
 For any $X,\ Y\in \Gamma(\mathcal{D}_1\oplus\{\xi\})$ and $Z\in \Gamma(\mathcal{D}_2)$, we have
 \begin{eqnarray*}
   g(\nabla_XY,Z) &=& g(\phi \bar{\nabla}_XY,\phi Z) \\
    &=& g(\bar{\nabla}_X\phi Y,\phi Z)-g((\bar{\nabla}_X\phi)Y,Z) \\
    &=& g(\bar{\nabla}_XP_1Y,\phi Z)+g(\bar{\nabla}_XQY,\phi Z)-\eta(Y)g(\phi X,\phi Z) \\
    &=& -g(\phi \bar{\nabla}_XP_1Y,Z)+g(\bar{\nabla}_XQY,P_2Z)+g(\bar{\nabla}_XQY,QZ)-\eta(Y)g(X,Z)\\
    &=& -g(\bar{\nabla}_XP_1^2Y,Z)-g(\bar{\nabla}_XQP_1Y,Z)+g((\bar{\nabla}_X\phi)P_1Y,Z)\\
    &&+g(\bar{\nabla}_XQY,P_2Z)-g(\bar{\nabla}_XQZ,\phi Y)+g(\bar{\nabla}_XQZ,P_1Y)-\eta(Y)g(X,Z)\\
    &=&\cos^2\theta_1g(\bar{\nabla}_XY,Z)-\sin2\theta_1X(\theta_1)g(Y,Z)-g(\bar{\nabla}_XQP_1Y,Z)  \\
    &&+g(\bar{\nabla}_XQY,P_2Z)+g(\bar{\nabla}_XbQZ,Y)+g(\bar{\nabla}_XcQZ,Y)\\
    &&-g((\bar{\nabla}_X\phi)QZ,Y)+g(\bar{\nabla}_XQZ,P_1Y)-\eta(Y)g(X,Z)
 \end{eqnarray*}
 Using (\ref{2.5}),  (\ref{2.7}), (\ref{2.12}), orthogonality of the distributions and symmetry of the shape operator,  the above equation reduces to
 \begin{eqnarray*}
      g(\nabla_XY,Z) &=& \cos^2\theta_1g(\bar{\nabla}_XY,Z)+g(A_{QP_1Y}Z,X)-g(A_{QY}P_2Z,X) \\
    &&+\sin^2\theta_2g(\bar{\nabla}_XY,Z)+g(A_{QP_2Z}Y,X)-g(A_{QZ}P_1Y,X),
 \end{eqnarray*}
 from which the relation (\ref{3.4}) follows. Also, for $X\in\Gamma(\mathcal{D}_1)$ and $Z,\ W\in\Gamma(\mathcal{D}_2)$, we have
 \begin{eqnarray*}
   g(\nabla_ZW,X) &=& g(\phi \bar{\nabla}_ZW,\phi X)+\eta(X)g(\bar{\nabla}_ZW,\xi) \\
    &=& g(\bar{\nabla}_Z\phi W,\phi X)-g((\bar{\nabla}_Z\phi)W,\phi X)-\eta(X)g(W,\bar{\nabla}_Z\xi) \\
    &=& g(\bar{\nabla}_ZP_2W,\phi X)+g(\bar{\nabla}_ZQW,\phi X)-\eta(X)g(Z,W) \\
    &=& -g(\phi \bar{\nabla}_ZP_2W,X)+g(\bar{\nabla}_ZQW,P_1X)\\
    &&+g(\bar{\nabla}_ZQW,QX)-\eta(X)g(Z,W)\\
    &=& -g(\bar{\nabla}_ZP_2^2W,X)-g(\bar{\nabla}_ZQP_2W,X)+g((\bar{\nabla}_Z\phi)P_2W,X)\\
    &&+g(\bar{\nabla}_ZQW,P_1X)-g(\bar{\nabla}_ZQX,\phi W)\\
    &&+g(\bar{\nabla}_ZQX,P_2W)-\eta(X)g(Z,W)\\
    &=&\cos^2\theta_2g(\bar{\nabla}_ZW,X)-\sin2\theta_2Z(\theta_2)g(W,X)-g(\bar{\nabla}_ZQP_2W,X)  \\
    &&+g(\bar{\nabla}_ZQW,P_1X)+g(\bar{\nabla}_ZbQX,W)+g(\bar{\nabla}_ZcQX,W)\\
    &&-g((\bar{\nabla}_Z\phi)QX,W)+g(\bar{\nabla}_ZQX,P_2W)-\eta(X)g(Z,W)
 \end{eqnarray*}
 Using (\ref{2.5}),  (\ref{2.7}), (\ref{2.12}), orthogonality of the distributions and symmetry of the shape operator, 
 the above equation reduces to
 \begin{eqnarray*}
      g(\nabla_ZW,X) &=& \cos^2\theta_2g(\bar{\nabla}_ZW,X)+g(A_{QP_2W}X,Z)-g(A_{QW}P_1X,Z) \\
    &&+\sin^2\theta_1g(\bar{\nabla}_ZW,X)+g(A_{QP_1X}W,Z)\\
    &&-g(A_{QX}P_2W,Z)-\eta(X)g(Z,W),
 \end{eqnarray*}
from which the relation (\ref{3.5}) follows.
 \end{proof}
 \begin{corollary}
  Let $M$ be a pointwise bi-slant submanifold of $\bar{M}$ with pointwise slant distributions $\mathcal{D}_1$ and $\mathcal{D}_2$ having distinct slant functions $\theta_1$ and $\theta_2$, respectively. Then the distribution $\mathcal{D}_1\oplus\{\xi\}$ defines a totally geodesic foliation if and only if
  \begin{equation}\label{3.6}
    g(A_{QP_2Z}X-A_{QZ}P_1X+A_{QP_1Y}Z-A_{QY}P_2Z,Y)=0,
  \end{equation}
  for any $X,\ Y\in\Gamma(\mathcal{D}_1\oplus\{\xi\})$ and $Z\in \Gamma(D_2)$.
 \end{corollary}
 \begin{proof}
 It follows from (\ref{3.4}).
 \end{proof}
 \begin{corollary}
   Let $M$ be a pointwise bi-slant submanifold of $\bar{M}$ with pointwise slant distributions $\mathcal{D}_1$ and $\mathcal{D}_2$ having distinct slant functions $\theta_1$ and $\theta_2$, respectively. Then the distribution $\mathcal{D}_2$ defines a totally geodesic foliation if and only if
   \begin{equation}\label{3.7}
   g(A_{QP_2W}X-A_{QW}P_1X,Z)+g(A_{QP_1X}W-A_{QX}P_2W,Z)=\eta(X)g(Z,W),
   \end{equation}
   for any $X\in\Gamma(\mathcal{D}_1\oplus\{\xi\})$ and $Z,\ W\in \Gamma(\mathcal{D}_2)$.
 \end{corollary}
 \begin{proof}
 It follows from (\ref{3.5})
 \end{proof}
 \section{Warped product pointwise bi-slant submanifolds}
 \begin{definition}\cite{BISHOP}
Let $(N_1,g_1)$ and $(N_2,g_2)$ be two Riemannian manifolds with Riemannian metric $g_1$
and $g_2$ respectively and $f$ be a positive definite smooth function on $N_1$. The warped product
of $N_1$ and $N_2$ is the Riemannian manifold $N_1\times_{f}N_2 = (N_1\times N_2,g)$, where
\begin{equation}
\label{4.1}
g=g_1+f^2g_2.
\end{equation}
\end{definition}
\noindent A warped product manifold $N_1\times_{f}N_2$ is said to be trivial if the warping function $f$ is constant.\\
\indent Let $M=N_1\times_{f}N_2$ be a warped product manifold, then we have \cite{BISHOP}
\begin{eqnarray}\label{4.2}
\nabla_UX = \nabla_XU = (X\ln f) U,
\end{eqnarray}
for any $X$, $Y\in\Gamma(TN_1)$ and $U\in\Gamma(TN_2)$.\\ Now we construct an example of warped product pointwise bi-slant submanifold of Kenmotsu manifold.
\begin{example}
We consider $\bar{M}=\mathbb{R}^{13}$ with the cartesian coordinates $(X_1,\ Y_1,\ \cdots,$\\$\ X_6,\ Y_6,\ t)$ and with its usual Kenmotsu structure
$(\phi, \xi, \eta, g)$, given by $$\eta=dt,\ \ \xi=\frac{\partial}{\partial t}, \ \ g
=\eta\otimes\eta+e^{2t}\{(dx^i\otimes dx^i+dy^i\otimes dy^i)\}$$ and
$$\phi\bigg(\displaystyle\sum_{i=1}^{6}(X_i\frac{\partial}{\partial x^i}+Y_i\frac{\partial}{\partial y^i})+t\frac{\partial}{\partial t}\bigg)=
\displaystyle\sum_{i=1}^{6}(-Y_i\frac{\partial}{\partial x^i}+X_i\frac{\partial}{\partial y^i}).$$\\
Now, we consider a submanifold of $\mathbb{R}^{13}$ defined by the immersion $\chi$ as follows:
\begin{eqnarray*}
\chi(u, v, \theta, \phi,\,t)=(u\cos\theta,\,v\cos\phi,\,u\sin\theta,\,v\sin\phi,\,u\cos\phi,\,v\cos\theta,\\
u\sin\phi,\,u\sin\theta,3\theta+2\phi,\,2\theta+3\phi,\,0,\,0,\,t).
\end{eqnarray*}
Then, it is easy to see that
\begin{eqnarray*}
  Z_1 &=& \frac{1}{e^t}\bigg(\cos \theta\frac{\partial}{\partial x_1}+\sin \theta\frac{\partial}{\partial x_2}+\cos \phi\frac{\partial}{\partial x_3}
  +\sin\phi\frac{\partial}{\partial x_4}\bigg), \\
  Z_2 &=& \frac{1}{e^t}\bigg(\cos\phi\frac{\partial}{\partial y_1}+\sin\phi\frac{\partial}{\partial y_2}+\cos \theta\frac{\partial}{\partial y_3}
  +\sin\theta\frac{\partial}{\partial y_4}\bigg), \\
  Z_3 &=&\frac{1}{e^t}\bigg( -u\sin \theta\frac{\partial}{\partial x_1}+u\cos \theta\frac{\partial}{\partial x_2}-v\sin \theta\frac{\partial}{\partial x_3}
  +v\cos\theta\frac{\partial}{\partial x_4}+3\frac{\partial}{\partial x_5}+2\frac{\partial}{\partial y_5}\bigg), \\
  Z_4 &=& \frac{1}{e^t}\bigg(-v\sin \phi\frac{\partial}{\partial y_1}+v\cos \phi\frac{\partial}{\partial y_2}-u\sin \phi\frac{\partial}{\partial y_3}
  +v\cos\phi\frac{\partial}{\partial y_4}+2\frac{\partial}{\partial x_5}+3\frac{\partial}{\partial y_5}\bigg),\\
  Z_5 &=&\frac{\partial}{\partial t}.
\end{eqnarray*}
form a local orthonarmal frame of $TM$. Also, we have
\begin{eqnarray*}
  \phi Z_1 &=& \frac{1}{e^t}\bigg(-\cos \theta\frac{\partial}{\partial y_1}-\sin \theta\frac{\partial}{\partial y_2}-\cos \phi\frac{\partial}{\partial y_3}
  -\sin\phi\frac{\partial}{\partial y_4}\bigg), \\
  \phi Z_2 &=&\frac{1}{e^t}\bigg( \cos\phi\frac{\partial}{\partial x_1}+\sin\phi\frac{\partial}{\partial x_2}+\cos \theta\frac{\partial}{\partial x_3}
  +\sin\theta\frac{\partial}{\partial x_4}\bigg), \\
  \phi Z_3 &=& \frac{1}{e^t}\bigg( u\sin \theta\frac{\partial}{\partial y_1}-u\cos \theta\frac{\partial}{\partial y_2}-v\sin \theta\frac{\partial}{\partial y_3}
  -v\cos\theta\frac{\partial}{\partial y_4}-3\frac{\partial}{\partial y_5}+2\frac{\partial}{\partial x_5}\bigg), \\
  \phi Z_4 &=&\frac{1}{e^t}\bigg(-v\sin \phi\frac{\partial}{\partial y_1}+v\cos \phi\frac{\partial}{\partial y_2}+u\sin \phi\frac{\partial}{\partial y_3}
  -v\cos\phi\frac{\partial}{\partial y_4}-2\frac{\partial}{\partial y_5}+3\frac{\partial}{\partial x_5}\bigg),\\
  \phi Z_5 &=& 0.
\end{eqnarray*}
If we define $\mathcal{D}_1=\{Z_1,\ Z_2\}$ and $\mathcal{D}_2=\{Z_3,\,Z_4\}$, then by simple calculations, we obtain
$g(\phi Z_1,Z_2)=2\cos(\theta-\phi)$ and $g(\phi Z_3,Z_4)=5$ and hence the distributions $\mathcal{D}_1$ and $\mathcal{D}_2$ are pointwise slant with slant
functions $\cos^{-1}[2\cos(\theta-\phi)]$ and $\cos^{-1}(\frac{5}{\sqrt{u^2+v^2+13}})$, respectively. Also, $\xi$ is tangent to $\mathcal{D}_1$.
 Consequently, $M$ is a proper pointwise bi-slant submanifold of $\mathbb{R}^{13}$. Also, it is clear that both the distributions $\mathcal{D}_1\oplus\{\xi\}$ and $\mathcal{D}_2$ are integrable. If we denote the integral manifolds of $\mathcal{D}_1\oplus\{\xi\}$ and
 $\mathcal{D}_2$ by $M_1$ and $M_2$, respectively then the metric tensor $g_{M}$ of $M$ is given by $$g_M=2(du^2+dv^2)+dt^2+(u^2+v^2+13)(d\theta^2+d\phi^2),$$ where $g_{M_1}=2(du^2+dv^2)+dt^2$ is the metric tensor of $M_1$ and
 $g_{M_2}=(u^2+v^2+13)(d\theta^2+d\phi^2)$ is the metric tensor of $M_2$. Thus $M=M_1\times_fM_2$ is a warped product pointwise bi-slant submanifold of Kenmotsu manifold with warping function $f=\sqrt{u^2+v^2+13}$.
\end{example}
Now we have the following lemmas:
\begin{lemma}
  Let $M=M_1\times_fM_2$ be a warped product pointwise bi-slant submanifold of a Kenmotsu manifold $\bar{M}$ with distinct slant functions $\theta_1$ and $\theta_2$ such that $\xi\in \Gamma(TM_1)$. Then
    \begin{eqnarray}
  \label{4.3}    g(h(X,W),QP_2Z)-g(h(X,P_2Z),QW)&=&\sin\theta_2X(\theta_2)g(Z,W)
  \end{eqnarray}
  and
  \begin{eqnarray}
   \label{4.4}g(h(X,Z),QW)-g(h(X,W),QZ)&=&\tan\theta_2X(\theta_2)g(P_2Z,W),
    \end{eqnarray}
    for any $X\in\Gamma(TM_1)$ and $Z,\ W\in\Gamma(TM_2)$.
\end{lemma}
\begin{proof}
 For any $X\in \Gamma(TM_1)$ and $Z,\ W\in \Gamma(TM_2)$, we have from (\ref{4.2}) that
 \begin{equation}\label{4.5}
 g(\bar{\nabla}_XZ,W)=g(\nabla_XZ,W)=(X\ln f)g(Z,W).
 \end{equation}
 Also, we have
 \begin{eqnarray*}
   g(\bar{\nabla}_XZ,W) &=& g(\phi\bar{\nabla}_XZ,\phi W)+\eta(W)g(\bar{\nabla}_XZ,\xi) \\
    &=&g(\bar{\nabla}_X\phi Z,\phi W)-g((\bar{\nabla}_X\phi)Z,W)  \\
    &=& g(\bar{\nabla}_XP_2Z,P_2W)+g(\bar{\nabla}_XP_2Z,QW)+g(\bar{\nabla}_XQZ,\phi W) \\
    &=& g(\bar{\nabla}_XP_2Z,P_2W)+g(\bar{\nabla}_XP_2Z,QW)-g(\bar{\nabla}_XbQZ,W)\\
    &&-g(\bar{\nabla}_XcQZ,W)+g(\bar{\nabla}_XQP_2Z,W).
 \end{eqnarray*}
 Using (\ref{2.5})-(\ref{2.7}) and (\ref{2.12}), the above equation yields
 \begin{eqnarray*}
   g(\bar{\nabla}_XZ,W) = g(\bar{\nabla}_XP_2Z,P_2W)+g(h(X,P_2Z),QW)\\
   +\sin^2\theta_2g(\bar{\nabla}_XZ,W)+\sin\theta_2X(\theta_2)g(Z,W)+g(h(X,W),QP_2Z).
 \end{eqnarray*}
 Using (\ref{4.2}) in the above equation, we obtain
 \begin{eqnarray}
 \label{4.6} g(\bar{\nabla}_XZ,W) = (X\ln f)g(P_2Z,P_2W)+g(h(X,P_2Z),QW) \\
  \nonumber +\sin^2\theta_2(X\ln f)g(Z,W)+\sin\theta_2X(\theta_2)g(Z,W)+g(h(X,W),QP_2Z).
 \end{eqnarray}
 Using (\ref{2.9}) and (\ref{4.5}), (\ref{4.6}) yields
 \begin{eqnarray}
 \label{4.7}(X\ln f)g(Z,W)=\cos^2\theta_2 X\ln fg(Z,W)+g(h(X,P_2Z),QW) \\
  \nonumber+\sin^2\theta_2(X\ln f)g(Z,W) +\sin\theta_2X(\theta_2)g(Z,W)+g(h(X,W),QP_2Z).
 \end{eqnarray}
 The relation (\ref{4.3}) follows from (\ref{4.7}) and (\ref{4.4}) obtained from (\ref{4.3}) by interchanging $Z$ by $P_2Z$.
\end{proof}
\begin{lemma}
  Let $M= M_1\times_fM_2$ be a warped product pointwise bi-slant submanifold of a Kenmotsu manifold $\bar{M}$ such that $\xi\in\Gamma(TM_1)$ with distinct slant functions $\theta_1,  \theta_2$. Then
  \begin{eqnarray}
  \label{4.8}  g(h(X,W),QP_2Z)-g(h(X,P_2Z),QW)=\\
  \nonumber 2\cos^2\theta_2\bigg\{(X\ln f)-\eta(X)\bigg\}g(Z,W),
  \end{eqnarray}
  for any $X\in\Gamma(TM_1)$ and $Z,\ W\in\Gamma(TM_2)$.
\end{lemma}
\begin{proof}
For any $X\in\Gamma(TM_1)$ and $Z,\ W\in\Gamma(TM_2)$, we have
\begin{eqnarray*}
  g(h(X,Z),QW) &=& g(\bar{\nabla}_ZX,QW) \\
   &=& g(\bar{\nabla}_ZX,\phi W)-g(\bar{\nabla}_ZX,P_2W)\\
   &=& -g(\bar{\nabla}_Z\phi X,W)+g((\bar{\nabla}_Z\phi)X,W)-g(\bar{\nabla}_ZX,P_2W)\\
   &=&-g(\bar{\nabla}_ZP_1X,W)-g(\bar{\nabla}_ZQX,W)\\
   &&-\eta(X)g(\phi Z,W)-g(\bar{\nabla}_ZX,P_2W)
\end{eqnarray*}
Using (\ref{4.2}) in the above equation, we get
\begin{eqnarray}
  \label{4.9} g(h(X,Z),QW) &=& -(P_1X\ln f)g(Z,W)+g(h(Z,W),QX) \\
\nonumber   &&+\bigg\{(X\ln f)-\eta(X)\bigg\}g(P_2Z,W).
\end{eqnarray}
By polarization, we obtain
\begin{eqnarray}
  \label{4.10}g(h(X,W),QZ) &=& -(P_1X\ln f)g(W,Z)+g(h(Z,W),QX) \\
  \nonumber && +\bigg\{(X\ln f)-\eta(X)\bigg\}g(P_2W,Z).
\end{eqnarray}
Subtracting (\ref{4.10}) from (\ref{4.9}), we get
\begin{eqnarray}
  \label{4.11}g(h(X,Z),QW)-g(h(X,W),QZ) \\
  \nonumber= 2[(X\ln f)-\eta(X)]g(P_2Z,W).
\end{eqnarray}
Interchanging $Z$ by $P_2Z$ in (\ref{4.11}), we get (\ref{4.8}).
\end{proof}
\begin{theorem}
There exists a proper warped product pointwiase bi-slant submanifold $M=M_1\times_fM_2$ of a Kenmotsu manifold $\bar{M}$, with distinct slant functions $\theta_1$ and $\theta_2$, if and only if $\tan\theta_2X(\theta_2)-\eta(X)\neq0$.
\end{theorem}
\begin{proof}
For any $X\in\Gamma(TM_1)$ and $Z,\ W\in \Gamma(TM_2)$ from (\ref{4.3}) and (\ref{4.8}) , we have
\begin{equation}\label{4.12}
2\cos^2\theta_2[(X\ln f)-\eta(X)]g(Z,W)=-\sin2\theta_2X(\theta_2)g(Z,W).
\end{equation}
Since $M$ is proper so $\theta_2\neq\frac{\pi}{2}$ and hence from (\ref{4.12}), we get $[(X\ln f)-\eta(X)+\tan\theta_2X(\theta_2)]g(Z,W)=0$,
which implies that $(X\ln f)=\eta(X)-\tan(\theta_2)X(\theta_2)$. This proves the theorem.
\end{proof}
\indent A warped product submanifold $M=M_1\times_fM_2$ of a Kenmotsu manifold $\bar{M}$ is known to be mixed totally geodesic if $h(X,Z)=0$,
for any $X\in \Gamma(TM_1)$ and $Z\in \Gamma(TM_2)$. \\
\begin{theorem}
Let $M=M_1\times_fM_2$ be a warped product pointwise bi-slant submanifold of a Kenmotsu manifold $\bar{M}$, withd istinct slant functions $\theta_1$ and $\theta_2$. $M$ is a mixed totally geodesic
warped product submanifold, then one of the following two cases holds:\\
(i) $\theta_2=\frac{\pi}{2}$, i.e., $M$ is a warped product pointwise pseudo slant submanifold of the
form $M_1\times_f M_\bot$, where $M_\bot$ is an anti-invariant submanifold $\bar{M}$.\\
(ii) or $(X\ln f)=\eta(X)$.
\end{theorem}
\begin{corollary}
  For a proper warped product mixed totally geodesic pointwise bi-slant submanifold $M=M_1\times f M_2$ of a Kenmotsu manifold $\bar{M}$ such that $\xi\in \Gamma(TM_1)$,
  $(\xi \ln f)=1$.
\end{corollary}
\begin{proof}
  If $M$ be mixed totally geodesic then from (\ref{4.8}) we get\\
  $\cos^2\theta_2[(X\ln f)-\eta(X)]g(Z,W)=0$. From which we get either $\cos^2\theta_2=0$ i.e., $\theta_2=\frac{\pi}{2}$
  or $(X\ln f)=\eta(X)$ and hence the theorem is proved.
\end{proof}
\begin{lemma}
Let $M=M_1\times_fM_2$ be a warped product pointwise bi-slant submanifold of a Kenmotsu manifold $\bar{M}$, such that $\xi$ is tangent to
 $M_1$, with slant functions $\theta_1$ and $\theta_2$. Then
 we have
\begin{equation}\label{4.13}
  g(h(X,Y),QZ)=g(h(X,Z),QY),
 \end{equation}
 \begin{eqnarray}\label{4.14}
  &&g(h(Z,W),QX)-g(h(X,Z),QW) \\
   \nonumber  && =(P_1X\ln f)g(Z,W) +[(X\ln f) -\eta(X)]g(Z,P_2W),
 \end{eqnarray}
 \begin{align}\label{4.15}
   &g(h(Z,W),QP_1X)-g(h(P_1X,Z),QW) \\
   \nonumber & =(P_1X\ln f)g(Z,P_2W) -\cos^2\theta_1[(X\ln f)-\eta(X)]g(Z,W),
 \end{align}
 \begin{align}\label{4.16}
    & g(h(Z,P_2W),QX)-g(h(X,Z),QP_2W) \\
   \nonumber &=(P_1X\ln f)g(Z,P_2W) -\cos^2\theta_2[(X\ln f)-\eta(X)]g(Z,W),
 \end{align}
 \begin{align}\label{4.17}
    & g(h(Z,W),QP_1X)-g(h(P_1X,Z),QW)+g(h(X,Z),QP_2W) \\
   \nonumber & -g(h(Z,P_2W),QX)=(\cos^2\theta_2-\cos^2\theta_1)[(X\ln f)-\eta(X)]g(Z,W).
 \end{align}
 for any $X,\ Y\in\Gamma(TM_1)$ and $Z,\ W\in \Gamma(TM_2)$.
\end{lemma}
\begin{proof}
 For any $X,\ Y\in\Gamma(TM_1)$ and $Z\in \Gamma(TM_2)$, we have
 \begin{eqnarray*}
 g(h(X,Y),QZ)&=& g(\bar{\nabla}_XY,\phi Z)-g(\bar{\nabla}_XY,P_2Z) \\
    &=&-g(\bar{\nabla}_X\phi Y,Z)+g((\bar{\nabla}_X\phi)Y,Z)+g(\bar{\nabla}_XP_2Z,Y).
 \end{eqnarray*}
 By virtue of (\ref{2.5}), (\ref{3.2}) yields
 \begin{eqnarray*}
   g(h(X,Y),QZ) &=& -g(\bar{\nabla}_XP_1Y,Z)-g(\bar{\nabla}_XQY,Z)-\eta(Y)g(\phi X,Z) \\
   &&+g(\bar{\nabla}_XP_2Z,Y)\\
   &=&-g(\bar{\nabla}_XP_1Y,Z)+g(h(X,Z),QY)-\eta(Y)g(\phi X,Z) \\
   &&+g(\bar{\nabla}_XP_2Z,Y)
 \end{eqnarray*}
 Using (\ref{4.2}) in the above equation, we get
 \begin{eqnarray}
 \label{4.18}g(h(X,Y),QZ) &=& (X\ln f)g(P_1Y,Z)+g(h(X,Z),QY) \\
   \nonumber && -\eta(Y)g(P_1X,Z)+(X\ln f)g(Y,P_2Z),
 \end{eqnarray}
 from which the relation (\ref{4.13}) follows. \\
Also  for any $X\in\Gamma(TM_1)$ and $Z,\ W\in \Gamma(TM_2)$, we have
\begin{eqnarray*}
  g(h(Z,W),QX) &=& g(\bar{\nabla}_ZW,\phi X)-g(\bar{\nabla}_ZW,P_1X) \\
   &=& -g(\bar{\nabla}_Z\phi W,X)+g((\bar{\nabla}_Z\phi)W,X)+g(\bar{\nabla}_ZP_1X,W)
\end{eqnarray*}
By virtue of (\ref{2.5}) and (\ref{3.2}), the above equation yields
\begin{eqnarray}
  \label{4.19}g(h(Z,W),QX) &=& -g(\bar{\nabla}_ZX,P_2W)-g(\bar{\nabla}_ZQW,X) \\
\nonumber   &&+\eta(X)g(\phi Z,W)+g(\bar{\nabla}_ZP_1X,W).
\end{eqnarray}
Using (\ref{2.7}) and (\ref{4.2}) in (\ref{4.19}), we get (\ref{4.14}). The relations (\ref{4.15}) and (\ref{4.16}) can be derived from
(\ref{4.14}) by replacing $X$ by $P_1X$ and $W$ by $P_2W$, respectively. Subtracting (\ref{4.16}) from (\ref{4.15}), we get (\ref{4.17}).
\end{proof}
Now, interchanging $W$ by $P_2W$ in (\ref{4.15}), we can derive
\begin{align}\label{4.20}
   &g(h(Z,P_2W),QP_1X)-g(h(P_1X,Z),QP_2W)  \\
  \nonumber & = -\cos^2\theta_2 (P_1X \ln f)g(Z,W)-\cos^2\theta_1\bigg\{(X\ln f)-\eta(X)\bigg\}g(Z,P_2W).
\end{align}
If we interchange $Z$ by $P_2Z$ in (\ref{4.14}) and (\ref{4.15}) then, we obtain
\begin{align}\label{4.21}
   & g(h(P_2Z,W),QX)-g(h(X,P_2Z),QW)\\
  \nonumber & =(P_1X \ln f)g(P_2Z,W) +\cos^2\theta_2[(X\ln f)-\eta(X)]g(Z,W)
\end{align}
and
\begin{align}\label{4.22}
   & g(h(P_2Z,W),QP_1X)-g(h(P_1X,P_2Z),QW) \\
  \nonumber & =\cos^2\theta_2(P_1X \ln f)g(Z,W) -\cos^2\theta_1[(X\ln f)-\eta(X)]g(P_2Z,W).
\end{align}
Interchanging $W$ by $P_2W$ in (\ref{4.21}) and (\ref{4.22}), we get
\begin{align}\label{4.23}
   & g(h(P_2Z,P_2W),QX)-g(h(X,P_2Z),QP_2W)  \\
  \nonumber & \cos^2\theta_2(P_1X \ln f)g(Z,W) -\cos^2\theta_2[(X\ln f)-\eta(X)]g(P_2Z,W).
\end{align}
and
\begin{eqnarray}\label{4.24}
   \qquad&& g(h(P_2Z,P_2W),QP_1X)-g(h(P_1X,P_2Z),QP_2W) \\
  \nonumber&& =\cos^2\theta_2(P_1X \ln f)g(Z,P_2W) -\cos^2\theta_1\cos^2\theta_2[(X\ln f)-\eta(X)]g(Z,W).
\end{eqnarray}
\section{characterization of warped product pointwise bi-slant submanifold}
In this section we found the characterization for a warped product pointwise bi-slant submanifold of Kenmotsu manifold.
\begin{theorem}
Let $M$ be a proper pointwise bi-slant submanifold of a Kenmotsu manifold $\bar{M}$ with pointwise slant distributions
$\mathcal{D}_1\oplus\{\xi\}$ and $\mathcal{D}_2$ . Then $M$ is locally a wraprd product submanifold of the form $M_1\times_fM_2$, where $M_1$
and $M_2$ are pointwise slant submanifolds with distinct slant functions $\theta_1$, $\theta_2$, if and only if the shape operator of $M$ satisfies
\begin{eqnarray}
  \label{5.1} A_{QP_1X}Z-A_{QZ}P_1X+A_{QP_2Z}X -A_{QX}P_2Z\\
 \nonumber =(\cos^2\theta_2-\cos^2\theta_1)[(X\mu)-\eta(X)] Z
\end{eqnarray}
for any $X\in\Gamma(\mathcal{D}_1\oplus\{\xi\})$ and $Z\in \Gamma(\mathcal{D}_2)$ and for some function $\mu$ on $M$ satisfying $(W\mu)=0$, for
any $W\in \Gamma(\mathcal{D}_2)$.
\end{theorem}
\begin{proof}
Let $M=M_1\times_fM_2$ be a warped product submanifold of $\bar{M}$. Then for any $X,\ Y\in\Gamma(TM_1)$ and $Z\in \Gamma(TM_2)$, we have
from (\ref{4.13}) that
\begin{equation}\label{5.2}
g(A_{QY}Z-A_{QZ}Y,X)=0.
\end{equation}
Interchanging $Y$ by $P_1Y$ in (\ref{5.2}), we obtain
\begin{equation}\label{5.3}
g(A_{QP_1Y}Z-A_{QZ}P_1Y,X)=0
\end{equation}
Again interchanging $Z$ by $P_2Z$, (\ref{5.2}) yields
\begin{equation}\label{5.4}
g(A_{QY}P_2Z-A_{QP_2Z}Y,X)=0.
\end{equation}
Subtracting (\ref{5.4}) from (\ref{5.3}), we get
\begin{eqnarray}
   \label{5.5} g(A_{QP_1X}Z-A_{QZ}P_1X+A_{QP_2Z}X -A_{QX}P_2Z,X)=0.
\end{eqnarray}
From (\ref{4.16}) and (\ref{5.5}), we get (\ref{5.1})\\
Conversely, let $M$ be a proper pointwise bi-slant subamnifold of $\bar{M}$. Then for any $X,\ Y\in \Gamma(\mathcal{D}_1\oplus\{\xi\})$
and $Z\in \Gamma(\mathcal{D}_2)$, we have from (\ref{3.4}) and (\ref{5.1}) that
\begin{eqnarray}
  \label{5.6}(\sin^2\theta_1-\sin^2\theta_)g(\nabla_XY,Z) = \\
\nonumber  (\cos^2\theta_2-\cos^2\theta_1)[(X\mu)-\eta(X)]g(X,Z) =0.
\end{eqnarray}
Since $\theta_1\neq\theta_2$ , the leaves of the distribution $\mathcal{D}_1\oplus\{\xi\}$ are totally geodesic in $M$.
Also, for any $X\in \Gamma(\mathcal{D}_1\oplus\{\xi\})$
and $Z,\ W\in \Gamma(\mathcal{D}_2)$, we have from (\ref{3.5}) and (\ref{5.1}) that
\begin{eqnarray}
 \label{5.7} (\sin^2\theta_2-\sin^2\theta_1)g(\nabla_ZW,X) = (\cos^2\theta_2-\cos^2\theta_1)\\
\nonumber \{(X\mu)-\eta(X)\}g(Z,W)-\eta(X)g(Z,W)
\end{eqnarray}
Using trigonometric identities on (\ref{5.7}), we get
\begin{equation}\label{5.8}
g(\nabla_ZW,X)=-(X\mu)g(Z,W).
\end{equation}
By polarization, we find
\begin{equation}\label{5.9}
g(\nabla_WZ,X)=-(X\mu)g(Z,W).
\end{equation}
 From (\ref{5.8}) and (\ref{5.9}), we obtain $g([Z,W],X)=0$ and hence the distribution $\mathcal{D}_2$ is integrable.
We consider a leaf $M_2$ of $\mathcal{D}_2$ and $h_2$ be the second fundamental form of $M_2$ in $M$. Then from (\ref{5.8}), we obtain
\begin{equation}\label{5.10}
g(h_2(Z,W),X)=g(\nabla_ZW,X)=-(X\mu)g(Z,W).
\end{equation}
Therefore, we get $h_2(Z,W)=-\boldsymbol{\nabla}\mu g(Z,W)$, where $\boldsymbol{\nabla}\mu$ is the gradient of $\mu$ and hence the leaf
 $M_2$ is totally umbilical in $M$ with mean curvature vector $H_2=-\boldsymbol{\nabla}\mu$. Since $W(\mu)=0$ for any
 $W\in \Gamma(\mathcal{D}_2)$, we can easily obtain $H_2$ is parallel corresponding to the normal connection $D^{\theta_2}$ of $M_2$ in $M$ \cite{UDD1}.
 Thus $M_2$ is an extrinsic sphere in $M$. Therefore by Theorem 2.2, we conclude that $M$ is locally a warped product submanifold.
 Thus the proof is complete.
\end{proof}
\section{An inequality for warped product pointwise bi-slant submanifold}
In this section, we establish a general sharp geometric inequality for a mixed totally geodesic proper pointwise bi-slant warped product
 submanifold of the form $M_1\times_fM_2$ such that $\xi\in\Gamma(TM_1)$ of a Kenmotsu manifold $\bar{M}$.
 \indent Let $M_1\times_fM_2$ be a warped product mixed totally geodesic proper pointwise bi-slant submanifold of a Kenmotsu manifold $\bar{M}$
 such that $\xi\in\Gamma(TM_1)$.
  Put $dim \bar{M}=2m+1$, $\dim M_1=2p+1$, $dim M_2=2q$ and $dim M=n=2p+2q+1$. Let $\mathcal{D}_1$ and $\mathcal{D}_2$ be the tangent bundles of $M_1$
  and $M_2$, respectively. Assume that \\
  $\{e_1,\cdots,e_p,e_{p+1}=\sec\theta_1P_1e_1,\cdots,e_{2p}=\sec \theta_1P_1e_p,e_{2p+1}=\xi\}$  is a local orthonormal frame of $\mathcal{D}_1$,\\
  $\{e_{2p+2}=e_1^*,\cdots,e_{2p+q+1}=e_q^*,e_{2p+q+2}=e_{q+1}^*=\sec\theta_2P_2e_1^*,\cdots,e_{2p+2q+1}=e_{2q}^*=\sec\theta_2P_2e_q^*\}$ is a local orthonormal frame of $\mathcal{D}_2$. Then\\ $\{e_{n+1}=\hat{e}_1=\csc\theta_1Qe_1,\cdots,e_{n+p}=\hat{e}_p=\csc\theta_1Qe_p,\\e_{n+p+1}=\hat{e}_{p+1}=\csc\theta_1\sec\theta_1QP_1e_1,\cdots,
e_{n+2p}=\hat{e}_{2p}=\csc\theta_1\sec\theta_1QP_1e_p\}$,\\ $\{e_{n+2p+1}=\tilde{e}_1=\csc\theta_2Qe^*_1,\cdots,
e_{n+2p+q}=\tilde{e}_q=\csc\theta_2Qe^*_q,e_{n+2p+q+1}=\tilde{e}_{q+1}
=\csc\theta_2\sec\theta_2QP_2e^*_1,\cdots,e_{2n-1}=\tilde{e}_{2q}=\csc\theta_2\sec\theta_2QP_2e^*_1\}$ and $\{e_{2n},\cdots,
e_{2m+1}\}$ are local orthonormal frames of $Q\mathcal{D}_1, \ Q\mathcal{D}_2$ and $\nu$ respectively.
\begin{theorem}
Let $M_1\times_fM_2$ be a warped product mixed totally geodesic proper pointwise bi-slant submanifold of a Kenmotsu manifold $\bar{M}$.
Then we have\\
(i) The squared norm of the second fundamental form $h$ of $M$ satisfies
\begin{eqnarray}\label{6.1}
 \qquad \|h\|^2 \geq 2q\csc^2\theta_1\bigg(\cos^2\theta_1+\cos^2\theta_2\bigg)\bigg\{\|\boldsymbol{\nabla}^{\theta_1}lnf\|^2-1-\sum_{r=1}^{p}(e_rlnf)^2\bigg\},
\end{eqnarray}
where $\boldsymbol{\nabla}^{\theta_1}\ln f$ denotes the gradient of $\ln f$ along $M_1$, $2q=$ dim $M_2$ and $\theta_1,\ \theta_2$ are the slant angles of $M_1$
and $M_2$, respectively.\\
(ii) If the equality sign of (\ref{6.1}) holds, then $M_1$ is totally geodesic and $M_2$ is totally umbilical in $\bar{M}$.
\end{theorem}
\begin{proof}
From the definition of $h$, we have
\begin{equation*}
\|h\|^2=\sum_{i,j=1}^{n}g(h(e_i,e_j),h(e_i,e_j))=\sum_{r=n+1}^{2m+1}\sum_{i,j=1}^{n}g(h(e_i,e_j),e_r)^2.
\end{equation*}
Now we decompose the above relation for the frames of $\mathcal{D}_1$ and $\mathcal{D}_2$, as follows
\begin{eqnarray}\label{6.2}
  \|h\|^2 &=& \sum_{r=n+1}^{2m+1}\sum_{i,j=1}^{2p+1}g(h(e_i,e_j),e_r)^2+2\sum_{r=n+1}^{2m+1}\sum_{i=1}^{2p+1}\sum_{j=1}^{2q}g(h(e_i,e^*_j),e_r)^2 \\
  \nonumber && +\sum_{r=n+1}^{2m}\sum_{i,j=1}^{2q}g(h(e^*_i,e^*_j),e_r)^2
\end{eqnarray}
Since $M$ is mixed totally geodesic, the second term in the right hand side of (\ref{6.2}) is zero and decomposing the remaining terms according to (\ref{3.3a}), we obtain
\begin{eqnarray}\label{6.3}
  \|h\|^2 &=& \sum_{r=n+1}^{n+2q}\sum_{i,j=1}^{2p+1}g(h(e_i,e_j),e_r)^2+\sum_{r=n+2q+1}^{2n-1}\sum_{i,j=1}^{2p+1}g(h(e_i,e_j),e_r)^2 \\
 \nonumber  && +\sum_{r=2n}^{2m+1}\sum_{i,j=1}^{2p+1}g(h(e_i,e_j),e_r)^2+\sum_{r=n+1}^{n+2p}\sum_{i,j=1}^{2q}g(h(e^*_i,e^*_j),e_r)^2  \\
  \nonumber &&+\sum_{r=n+2p+1}^{2n-1}\sum_{i,j=1}^{2q}g(h(e^*_i,e^*_j),e_r)^2+\sum_{r=2n}^{2m+1}\sum_{i,j=1}^{2q}g(h(e^*_i,e^*_j),e_r)^2.
\end{eqnarray}
Next by removing the $\nu$ components in (\ref{6.3}), we have
\begin{eqnarray}\label{6.4}
  \|h\|^2\geq  && \sum_{r=1}^{2q}\sum_{i,j=1}^{2p+1}g(h(e_i,e_j),\widetilde{e}_r)^2+\sum_{r=1}^{2p}\sum_{i,j=1}^{2p+1}g(h(e_i,e_j),\hat{e}_r)^2 \\
  \nonumber &&+ \sum_{r=1}^{2p}\sum_{i,j=1}^{2q}g(h(e^*_i,e^*_j),\widetilde{e}_r)^2+\sum_{r=1}^{2q}\sum_{i,j=1}^{2q}g(h(e^*_i,e^*_j),\hat{e}_r).
\end{eqnarray}
Since we could not find any relation for $g(h(e_i,e_j),\hat{e}_r)$ for any $i,j=1,2,\cdots,2p+1$, $r=1,2,\cdots,2p$. And $g(h(e^*_i,e^*_j),\widetilde{e}_r)$, for any $i,j,r=1,2,\cdots,2q$ therefore, we leave the $2$nd and $4$th terms and get
\begin{equation}\label{6.5}
 \|h\|^2\geq   \sum_{r=1}^{2q}\sum_{i,j=1}^{2p+1}g(h(e_i,e_j),\widetilde{e}_r)^2+\sum_{r=1}^{2p}\sum_{i,j=1}^{2q}g(h(e^*_i,e^*_j),\hat{e}_r)^2.
\end{equation}
In view of (\ref{4.13}), (\ref{6.5}) yields
\begin{equation}\label{6.6}
  \|h\|^2\geq  \sum_{r=1}^{2q}\sum_{i,j=1}^{2p+1}g(h(e_i,e^*_j),\tilde{e}_r)^2+\sum_{r=1}^{2p}\sum_{i,j=1}^{2q}g(h(e^*_i,e^*_j),\hat{e}_r)^2.
\end{equation}
Since $M$ is mixed totally geodesic, we have
\begin{equation}\label{6.6a}
g(h(e_i,e_j),e_r)=0,
\end{equation}
 for every $i,\ j=1,\cdots,2p+1,\ j=n+1,\cdots,2q$.\\
 By virtue of (\ref{6.6a}),we have from (\ref{6.6}) that
\begin{equation}\label{6.7}
 \|h\|^2\geq  \sum_{r=1}^{2p}\sum_{i,j=1}^{2q}g(h(e^*_i,e^*_j),\hat{e}_r)^2.
\end{equation}
Thus, by using the orthonormal frame fields of $Q\mathcal{D}_1$ and $Q\mathcal{D}_2$, we derive
\begin{eqnarray}
\label{6.8}
  &&\|h\|^2 \geq \csc^2\theta_1\sum_{r=1}^{p}\sum_{i,j=1}^{q}g(h(e^*_i,e^*_j),Qe_r)^2\\
\nonumber&&+\csc^2\theta_1\sec^2{\theta}_2\sum_{r=1}^{p}\sum_{i,j=1}^{q}g(h(P_2e^*_i,e^*_j),Qe_r)^2 \\
\nonumber&& +\csc^2\theta_1\sec^2{\theta}_2\sum_{r=1}^{p}\sum_{i,j=1}^{q}g(h(e^*_i,P_2e^*_j),Qe_r)^2 \\
\nonumber&& +\csc^2\theta_1 \sec^4{\theta}_2\sum_{r=1}^{p}\sum_{i,j=1}^{q}g(h(P_2e^*_i,P_2e^*_j),Qe_r)^2\\
\nonumber&& + \csc^2\theta_1 \sec^2{\theta}_1\sum_{r=1}^{p}\sum_{i,j=1}^{q}g(h(e^*_i,e^*_j),QP_1e_r)^2 \\
\nonumber&& +\csc^2\theta_1\sec^2{\theta}_1\sec^2{\theta}_2\sum_{r=1}^{p}\sum_{i,j=1}^{q}g(h(P_2e^*_i,e^*_j),QP_1e_r)^2\\
\nonumber&&+\csc^2\theta_1\sec^2{\theta}_1\sec^2{\theta}_2\sum_{r=1}^{p}\sum_{i,j=1}^{q}g(h(Pe^*_i,P_2e^*_j),QP_1e_r)^2\\
\nonumber&& +\csc^2\theta_1\sec^2{\theta}_1\sec^4{\theta}_2\sum_{r=1}^{p}\sum_{i,j=1}^{q}g(h(P_2e^*_i,P_2e^*_j),QP_1e_r)^2.
\end{eqnarray}
Using (\ref{4.14})-(\ref{4.16}), (\ref{4.20})-(\ref{4.24}) and (\ref{6.6a}) in (\ref{6.8}), we get
\begin{eqnarray}\label{6.9}
  \|h\|^2\geq &&\csc^2\theta_1\sum_{r=1}^{p}\sum_{i,j=1}^{q}(P_1e_rlnf)^2g(e^*_i,e^*_j)^2 \\
  \nonumber && +\csc^2\theta_1\sum_{r=1}^{p}\sum_{i,j=1}^{q}(P_1e_rlnf)^2g(e^*_i,e^*_j)^2 \\
 \nonumber  &&+\csc^2\theta_1\sec^2\theta_1\cos^2\theta_2\sum_{r=1}^{p}\sum_{i,j=1}^{q}(P_1e_rlnf)^2g(e^*_i,e^*_j)^2  \\
 \nonumber  &&+\csc^2\theta_1\sec^2\theta_1\cos^2\theta_2\sum_{r=1}^{p}\sum_{i,j=1}^{q}(P_1e_rlnf)^2g(e^*_i,e^*_j)^2  \\
 \nonumber  &=&2q\csc^2\theta_1[1+\sec^2\theta_1\cos^2\theta_2]\sum_{r=1}^{p}(P_1e_rlnf)^2.
\end{eqnarray}
Now,
\begin{eqnarray*}
  \|\boldsymbol{\nabla}^{\theta_1}\ln f\|^2 &=& \sum_{r=1}^{2p+1}(e_rlnf)^2 \\
  \nonumber &=& \sum_{r=1}^{p}(e_rlnf)^2+\sum_{r=1}^{p}(\sec\theta_1 P_1e_rlnf)^2+(\xi lnf)^2 \\
  \nonumber &=& \sum_{r=1}^{p}(e_rlnf)^2+\sec^2\theta_1\sum_{r=1}^{p}(P_1e_rlnf)^2+1.\\
  \nonumber\ \ \text{as}\ \ (\xi ln f)=1
\end{eqnarray*}
From the above relation, we get
\begin{equation}\label{6.10}
  \sum_{r=1}^{p}(P_1e_rlnf)^2=\cos^2\theta_1\left[\|\boldsymbol{\nabla}^{\theta_1}lnf\|^2-1-\sum_{r=1}^{p}(e_rlnf)^2\right].
\end{equation}\label{6.12}
Using (\ref{6.10}) in (\ref{6.9}), we get (\ref{6.1}).\\ If equality of (\ref{6.1}) holds, then from the leaving third term of (\ref{6.3}),
we get $g(h(\mathcal{D}_1,\mathcal{D}_1),\nu)=0$, which implies that $h(\mathcal{D}_1,\mathcal{D}_1)\bot\nu$, i.e.,
\begin{equation}\label{6.11}
 h(\mathcal{D}_1,\mathcal{D}_1)\subset Q\mathcal{D}_1\oplus Q\mathcal{D}_2.
 \end{equation}
 and from the leaving second term in (\ref{6.4}), we find
 $g(h(\mathcal{D}_1,\mathcal{D}_1),Q\mathcal{D}_1)=0$, which implies that $h(\mathcal{D}_1,\mathcal{D}_1)\bot Q\mathcal{D}_1$, i.e.,
 \begin{equation}
 h(\mathcal{D}_1,\mathcal{D}_1)\subset Q\mathcal{D}_2\oplus\nu.
 \end{equation}
 Also, from (\ref{4.13}) and (\ref{6.6a}), we get $h(\mathcal{D}_1,\mathcal{D}_1)\bot Q\mathcal{D}_2$, i.e.,
 \begin{equation}\label{6.13}
 h(\mathcal{D}_1,\mathcal{D}_1)\subset Q\mathcal{D}_1\oplus\nu.
 \end{equation}
 From (\ref{6.11})-(\ref{6.13}), we conclude that
 \begin{equation}\label{6.14}
 h(\mathcal{D}_1,\mathcal{D}_1)=0.
 \end{equation}
  Since $M_1$ is totally geodesic in $M$ (\cite{BISHOP}, \cite{CHENS}), from (\ref{6.14}), we can say that $M_1$ is totally geodesic in $\bar{M}$.\\
  Similarly, for the leaving sixth term in (\ref{6.3}), we get $g(h(\mathcal{D}_2,\mathcal{D}_2),\nu)=0$, which implies that
  $h(\mathcal{D}_2,\mathcal{D}_2)\bot \nu$, i.e.,
  \begin{equation}\label{6.15}
  h(\mathcal{D}_2,\mathcal{D}_2)\subset Q\mathcal{D}_2\oplus \nu.
  \end{equation}
  Also, for the leaving fourth term in (\ref{6.4}), we obtain $g(h(\mathcal{D}_2,\mathcal{D}_2),Q\mathcal{D}_2)=0$, which implies that
  $h(\mathcal{D}_2,\mathcal{D}_2)\bot Q\mathcal{D}_2$ i.e.,
  \begin{equation}\label{6.16}
  h(\mathcal{D}_2,\mathcal{D}_2)\subset Q\mathcal{D}_1.
  \end{equation}
  From (\ref{6.15}) and (\ref{6.16}), we get
  \begin{equation}\label{6.17}
    h(\mathcal{D}_2,\mathcal{D}_2)\subset Q\mathcal{D}_1.
  \end{equation}
  Moreover, using (\ref{6.6a}) in (\ref{6.14}), we obtain
  \begin{equation}\label{6.18}
  g(h(Z,W),QX)=(P_1X\ln f)g(Z,W)+\{(X\ln f)-\eta(X)\}g(Z,P_2W),
  \end{equation}
  for any $X\in \Gamma(TM_1)$ and $Z,\ W\in \Gamma(TM_2)$.
  By Polarization of (\ref{6.18}), we obtain
  \begin{equation}\label{6.19}
  g(h(Z,W),QX)=(P_1X\ln f)g(Z,W)+[(X\ln f)-\eta(X)]g(W,P_2Z).
  \end{equation}
  Subtracting (\ref{6.19}) from (\ref{6.18}), we get
  \begin{equation}\label{6.20}
  g(h(Z,W),QX)=(P_1X\ln f)g(Z,W).
  \end{equation}
  From (\ref{6.17}), (\ref{6.20}) and the fact that $M_2$ is totally umbilical in $M$ (\cite{BISHOP}, \cite{CHENS}), we can conclude that $M_2$ is totally umbilical
  in $\bar{M}$. Thus the proof of the theorem is complete.
\end{proof}
\noindent\textbf{Acknowledgement:} The first two authors (S. K. Hui and J. Roy) gratefully acknowledges to the SERB (Project No: EMR/2015/002302), Govt. of India for financial assistance of the work.

\vspace{0.1in}
\noindent Shyamal Kumar Hui\\
Department of Mathematics, The University of Burdwan,\\
Burdwan, 713104, West Bengal, India.\\
E-mail: skhui@math.buruniv.ac.in.\\

\vspace{0.1in}
\noindent Joydeb Roy\\
Department of Mathematics, The University of Burdwan,\\
Burdwan, 713104, West Bengal, India.\\
E-mail:joydeb.roy8@gmail.com.\\

\vspace{0.1in}
\noindent Tanumoy Pal\\
Department of Mathematics, The University of Burdwan,\\
Burdwan, 713104, West Bengal, India.\\
E-mail: tanumoypalmath@gmail.com.\\

\begin{thebibliography}{1}
\bibitem{ATCE1}
M. Atceken, \emph{Warped product semi-slant submanifolds in Kenmotsu manifolds}, Turk. J. Math.,
{\bf 34} (2010), 425--432.
\bibitem{SOLAMY}
F. R. Al-Solamy and M. A. Khan, \emph{Pseudo-slant warped product submanifolds of Kenmotsu manifolds}, Mathematica Moravica,
{\bf 17} (2013), 51--61.
\bibitem{OTHMAN}
A. Ali, W. A. M. Othman and C. \"{O}zel, \emph{Some ineqalities of warped product pseudo-slant submanifolds of nearly Kenmotsu manifolds}, J. of Inequalities and Appl., (2015).
\bibitem{BEJ}
A. Bejancu, \emph{Geometry of CR-submanifolds}, D. Reidel Publ. Co., Dordrecht, Holland, {\bf 1986}.
\bibitem{BLAIR}
D. E. Blair, \emph{Contact manifolds in Riemannian geometry}, Lecture Notes in Math. {\bf 509}, Springer-Verlag, 1976.
\bibitem{BISHOP}
R. L. Bishop and B. O'Neill, \emph{Manifolds of negative
curvature}, Trans. Amer. Math. Soc., {\bf 145} (1969), 1--49.
\bibitem{CAR1}
J. L. Cabrerizo,  A. Carriazo, L. M. Fernandez and  M. Fernandez,
\emph{Semi-slant submanifolds of a Sasakian manifold}, Geom.
Dedicata, {\bf 78} (1999), 183-199.
\bibitem{CHENS}
B. Y. Chen, \emph{Slant immersions}, Bull. Austral. Math. Soc., {\bf 41} (1990), 135--147.
\bibitem{CHENCR1}
B. Y. Chen, \emph{Geometry of warped product CR-submanifolds in Kaehler manifold}, Monatsh. Math., {\bf 133} (2001),
177-195.
\bibitem{CHENCR2}
B. Y. Chen, \emph{Geometry of warped product CR-submanifolds in Kaehler manifolds II}, Monatsh. Math., {\bf 134} (2001),
103-119.
\bibitem{CHENP}
B. Y. Chen, \emph{Pseudo-Riemannian geometry,  $\delta$-invariants and applications}, World Scientific, Hackensack, NJ, 2011.
\bibitem{CHENBOOK}
B. Y. Chen, \emph{Differential geometry of warped product manifolds and submanifolds}, World Scientific, Hackensack, NJ, 2017.
\bibitem{CHENUD}
B. Y. Chen, S, \emph{Warped product pointwise bi-slant submanifolds of Kaehler manifold}, arXiv:1711.07117v1, [math.DG](2017)
\bibitem{ETAYO}
F. Etayo, \emph{On quasi-slant submanifolds of an almost Hermitian manifold}, Publ. Math. Debrecen, {\bf 53} (1998), 217--223.
\bibitem{HASE}
I. Hasegawa and I. Mihai, \emph{Contact CR-warped product submanifolds in Sasakian manifolds}, Geom. Dedicata,
{\bf 102} (2003), 143-150.
\bibitem{HIPKO}
Hiepko, S., \emph{Eine inner kennzeichungder verzerrten produkte}, Math. Ann. {\bf 241} (1979), 209--215.
\bibitem{HUI1}
S. K. Hui, \emph{On weakly $\phi$-symmetric Kenmotsu manifolds}, Acta Univ.  Palac.  Olom., Fac. Rer. Nat., Math.,  {\bf 51(1)} (2012), 43-50.
\bibitem{HUI2}
S. K. Hui, \emph{On $\phi$-pseudo symmetric Kenmotsu manifolds}, Novi Sad J. Math., {\bf 43(1)} (2013), 89-98.
\bibitem{HUI3}
S. K. Hui, \emph{On $\phi$-pseudo symmetric Kenmotsu manifolds with respect to quarter-symmetric metric connection}, Applied Sciences, {\bf 15} (2013), 71-84.
\bibitem{HAN}
S. K. Hui, M. Atceken and S. Nandy, \emph{Contact CR-warped product submanifolds of $(LCS)_n$-manifolds},
 Acta Math. Univ. Comenianae, {\bf 86} (2017), 101-109.
\bibitem{HAP}
S. K. Hui, M. Atceken and T. Pal, \emph{Warped product pseudo slant submanifolds $(LCS)_n$-manifolds},
 New Trends in Math. Sci., {\bf 5} (2017), 204--212.
\bibitem{HUOM}
S. K. Hui, S. Uddin, C. \"{O}zel and A. A. Mustafa, \emph{Warped product submanifolds of LP-Sasakian manifolds}, Hindwai Publ. Corp.,
Discrete Dynamics in Nature and Society, vol. 2012, doi:10.1155/2012/868549.
\bibitem{KEN}
K. Kenmotsu, \emph{A class of almost contact Riemannian manifolds}, Tohoku Math. J., {\bf 24} (1972), 93-103.
\bibitem{6}
 V. A. Khan and M. A. Khan, \emph{Pseudo-slant Submanifolds of a
Sasakian Manifold}, Indian J. Pure Appl. Math., {\bf 38} (2007), 31--42.
\bibitem{KHAN}
V. A. Khan, M. A. Khan and S. Uddin, \emph{Contact CR-warped product submanifolds of Kenmotsu
manifold}, Thai J. Math., {\bf 6} (2008), 307-314.
\bibitem{KHANS}
V. A. Khan and M. Shuaib, \emph{Pointwise pseudo-slant submanifolds of Kenmotsu manifold}, Filomat, {\bf 31} (2017), 5833--5853.
\bibitem{KHANS1}
V. A. Khan and M. Shuaib, \emph{Some warped product submanifolds of a Kenmotsu manifold}, Bull. Korean Math. Soc.,
{\bf 51} (2014), 863--881.
\bibitem{LOTTA}
 A. Lotta, \emph{Slant submanifolds in contact geometry}, Bull.
Math. Soc. Sci. Math. R. S. Roumanie, {\bf 39} (1996), 183--198.
\bibitem{MIUD}
I. Mihai and S. Uddin, \emph{Warped product pointwise semi-slant submanifolds of Sasakian manifolds}, arXiv:1706.04305v1, [math.DG.]
\bibitem{MURA}
C. Murathan, K. Arslan, R. Ezentas and I. Mihai, \emph{Contact CR-warped product submanifolds in
Kenmotsu space forms}, J. Korean Math. Soc., {\bf 42} (2005), 1101-1110.
\bibitem{UDD2}
A. Mustafa, A. De and S. Uddin, \emph{Characterization of warped product submanifolds in Kenmotsu manifolds}, Balkan J. of Geom. and its Appl.,
{\bf 20} (2015), 74--85.
\bibitem{UD1}
A. Mustafa, S. Uddin, V. A. Khan, and B. R. Wong, \emph{Contact CR-warped product submanifolds of nearly trans
Sasakian manifolds}, Taiwanese J. of Math., {17} (2013), 1473--1486.
\bibitem{PGS}
P. K. Pandey,  R. S. Gupta and A. Sharfuddin, \emph{B. Y. Chen's inequalities for bi-slant submanifolds in Kenmotsu space form}, Demonstratio
 Mathematica, {\bf 43} (2010), 887--898.
\bibitem{PARK}
K. S. Park, \emph{Pointwise slant and pointwise semi slant submanifolds almost contact metric manifold}, arXiv:1410.5587v2 [math.DG](2014).
\bibitem{TANNO}
S. Tanno, \emph{The automorphism groups of almost contact Riemannian manifolds}, Tohoku Math. J.,
{\bf 21} (1969), 21-38.
\bibitem{UDD1}
S. Uddin, \emph{Geometry of warped product semi-slant submanifolds of a Kenmotsu manifold}, Turk. J. Math.,
{\bf 36} (2012), 319--330.
\bibitem{UKK}
 S. Uddin, V. A. Khan and  K. A. Khan, \emph{Warped product submanifolds of Kenmotsu manifolds}, Turk J. Math., {\bf36}(2012), 319-330.
\bibitem{UOO}
S. Uddin, W. A. M. Othman, C. Ozel and A. Ali, \emph{A characterization and an improved inequality for
Warped product submanifolds in Kenmotsu manifold}, arXiv:1404.7258v2, [math.DG](2017).
\end{thebibliography}
\end{document}